\documentclass
[12pt]{article} \oddsidemargin 0in \evensidemargin 0in \textheight
9in \textwidth 6.5in \topmargin 0in \headheight 0in
\usepackage{amsmath,amsfonts,amsthm,amssymb}

\newtheorem{thm}{Theorem}
\newtheorem{theorem}{Theorem}[section]
\newtheorem{cor}[theorem]{Corollary}

\newtheorem{lemma}[theorem]{Lemma}

\newtheorem{defn}[theorem]{Definition}

\def\irr#1{{\rm Irr}(#1)}
\def\irrr#1#2 {\irr {#1 \mid #2}}

\title{Camina Triples}

\author {Nabil M.\ Mlaiki
   \\ {\it Department of Mathematical Sciences, Kent State University}
   \\ {\it Kent, Ohio 44242}
   \\ E-mail: nmlaiki2012@gmail.com
      }

\date{}

\begin{document}

\maketitle{}
\begin{abstract}
In this paper, we study Camina triples. Camina triples are a
generalization of Camina pairs. Camina pairs were first introduced
in 1978 by A.R. Camina in \cite{camina1}. Camina's work in
\cite{camina1} was inspired by the study of Frobenius groups. We
show that if $(G,N,M)$ is a Camina triple, then either $G/N$ is a
$p$-group, or $M$ is abelian, or $M$ has a non-trivial nilpotent or
Frobenius quotient.
\end{abstract}

\section{Introduction}

In this paper, we study Camina triples. Camina triples are a
generalization of Camina pairs. Camina pairs were first introduced
in 1978 by A.R. Camina in \cite{camina1}. Camina's work in
\cite{camina1} was inspired by the study of Frobenius groups. In
this paper we the notation as in \cite{Isaacs1}.

Throughout this paper, we say that $(G,N)$ is a Camina pair when $N$
is a normal subgroup of a group $G,$ and for all $x\in G\setminus
N,$ $x$ is conjugate to all of $xN.$ Chillag and Macdonald proved in
\cite{chillag2} two equivalent conditions of a pair $(G,N)$ to be a
Camina pair. They showed that if $(G,N)$ is a Camina pair, then for
every $x\in G\setminus N$ we have
$|{\bf{C}}_{G}(x)|=|{\bf{C}}_{G/N}(xN)|.$ Also, they proved that if
$(G,N)$ is a Camina pair, then for all $x\in G\setminus N$ and $z\in
N,$ there exists an element $y\in G$ so that $ [x,y]=z.$ In
\cite{MacDonald1}, MacDonald showed that if $(G,N)$ is a Camina pair
where $G$ is a $p$-group, then $N$ is a term in both the lower and
the upper central series. As was proved in \cite{chillag2}, if $\chi
\in \rm{Irr}(G)$ where $N\nleq ker(\chi),$ then $\chi$ vanishes on
$G\setminus N.$ Camina proved in \cite{camina1} that if $(G,N)$ is a
Camina pair, then either $N$ is a $p$-group or $G/N$ is a $p$-group
for some prime $p,$ or $G$ is Frobenius group with kernel $N.$ In
our first theorem, we prove some facts about the subgroup $M$ when
$(G,N,M)$ is a Camina triple.

First, define Irr$(G\mid M)=\{ \chi\in$ Irr$( G) \mid M \nleq
ker(\chi)\}.$
\begin{defn}
Let $1 < M \leq N $ be two nontrivial
normal subgroups of a finite group $G$. We say that $(G,N,M)$ is a
\emph{Camina Triple} if for every $g \in G \setminus N$, $g$ is a
conjugate to all of $gM$.
\end{defn}
 Notice that Camina pairs are special cases
of Camina triples when $ M = N$.

\begin{thm}
If $(G,N, M)$ is a Camina triple then the following are true:
\begin{enumerate}
\item $M$ is solvable.
\item $M$ has a normal $\pi$-complement $Q$ with $M/Q$ is nilpotent,
where $\pi$ is the set of primes that divide $\mid G: N\mid.$
\item If $x\in M,$ then there exists $\chi \in \rm{Irr}(G\mid M)$
such that $\chi(x)\neq 0.$
\item If $x\in G\setminus N,$ then $\chi(x)=0$ for all $\chi \in \rm{Irr}(G\mid M).$

\end{enumerate}
\end{thm}
The following collorary is an immediate consequence of Theorem 1.
\begin{cor}
If $(G,N)$ is a Camina pair and $x\in N,$ then there exist $\chi \in
\rm{ Irr}(G\mid N)$ such that $\chi(x) \neq 0.$
\end{cor}

Lewis showed in \cite{Lewis1} that if $V(G)<G,$ then for every $x\in
G\setminus V(G)$ we have $cl(x)=xG'.$ Thus, if $V(G)<G,$ the triple
$(G,V(G), G')$ is a Camina triple. So, by Theorem 1, we deduce that
$G'$ is solvable. Therefore, $G$ is solvable.

Our second theorem, which we consider the main result of this paper.

\begin{thm}
If $(G,N,M)$ is a Camina triple, then at least one of the following
hold:
\begin{enumerate}
\item $G/N$ is a $p$-group.
\item $M$ has a non-trivial nilpotent quotient.
\item $M$ has a non-trivial Frobenius quotient with an Frobenius complement that is an elementary abelain $p$-group.
\item $M$ is abelian.
\end{enumerate}
\end{thm}

In closing, we prove some facts about Camina pairs using Camina
triples results, and given the fact that they are special cases of
Camina triples. In \cite{camina1}, Camina defined a different
hypothesis that is equivalent to Camina pairs. Let $G$ be a finite
group with a proper normal subgroup $N \neq 1$ and a set of
irreducible non-trivial characters of $G$, $A=\{
\chi_{1}\cdots\chi_{n}\},$ where $n$ is a natural number, such that
\begin{enumerate}
\item $\chi_{i}$ vanishes on $ G\setminus N$ and
\item there exist natural numbers $\alpha_{1}\cdots\alpha_{n} > 0$
such that $\sum_{i=1}^{n}\alpha_{i}\chi_{i}$ is constant on
$N\setminus \{1\}.$
\end{enumerate}
We are able to identify the characters in Camina hypothesis in
\cite{camina1}. First, let $N$ be a normal subgroup of $G$ and
$\theta \in \rm{Irr}(N).$ The inertia group of $\theta$ in $G$
denoted by $T$ and defined by $\{g\in G\mid \theta^{g}=\theta\}.$
\begin{thm} \label{sixseventeen} Let
$(G,N) $ be a Camina Pair then, $A = \rm{Irr}(G|N)$.
\end{thm}

Our last theorem, states some new conditions for a pair $(G,N)$ to
be a Camina pair.
\begin{thm} \label{sixeighteen}
Let $G$ be a finite group and $ N\vartriangleleft G$, then the
following are equivalent:
\begin{enumerate}
\item $(G,N)$ is a Camina pair.
\item $V(G|N)$ = $N$.
\item There is no $x$ in $N$ such that $\chi (x) =0 $ for all $
\chi$'s in $ \rm{Irr}(G|N)$, and if $ x \in G\setminus N$, then
$\chi(x) =0 $ for all $ \chi$'s in $ \rm{Irr}(G|N)$
\end{enumerate}
\end{thm}

\section{ Camina triples }

In this section, we prove Theorem 1 and 2 along with some facts
about Camina triples. First, we prove some equivalent conditions for
a triple $(G,N,M)$ to be a Camina triple.

\begin{theorem} \label{sixtwo}
If $ 1 \neq M < N $ are two normal subgroups of a finite group $G$,
then the following are equivalent:
\begin{enumerate}
\item $(G,N,M)$ is  a Camina triple.
\item $|C_{G}(g)|=|C_{G/M}(Mg)|$ for
every $g\in G\setminus N.$
\item For every $g\in G\setminus N,$
we have $\chi(g)=0$ for all $\chi \in \rm{Irr}(G\mid M).$
\item $ V(G|M)\leq N.$
\item For all $g\in G\setminus N$ and $ z\in M,$ there exists $ y \in G$ such that
$[g,y] = z.$

\end{enumerate}
\end{theorem}
\begin{proof}
First, we show that (1) implies (2). Assume that $(G,N,M)$ is a
Camina triple and let $g\in G\setminus N.$ Notice that $cl(g)=
\cup_{x\in G} (Mg)^{x}$. Hence, $|G:C_{G}(g)|=|G/M:C_{G/M}(Mg)||M|,$
and so $|C_{G}(g)|=|C_{G/M}(Mg)|$ as desired. We now show that (2)
implies (3). Assume (2) and let $g\in G\setminus N.$ By the Second
Orthogonality Relation, we have
$$|{\bf{C}}_{G}(g)|=\sum_{\chi \in
\rm{Irr}(G)}|\chi(g)|^{2}=\sum_{\chi\in \rm{Irr}(G\mid
M)}|\chi(g)|^{2}+\sum_{\chi \in \rm{Irr}(G/M)}|\chi(g)|^{2}.$$ But
we know by (2) that $|C_{G}(g)|=|C_{G/M}(Mg)|=\sum_{\chi \in
\rm{Irr}(G/M)}|\chi(g)|^{2}.$ Hence, we obtain $\sum_{\chi\in
\rm{Irr}(G\mid M)}|\chi(g)|^{2}=0.$ Since $|\chi(g)|^{2}\ge 0$ for
all $\chi\in \rm{Irr}(G\mid M),$ we deduce that $\chi(g)=0$ for all
$\chi\in \rm{Irr}(G\mid M).$ Next, we prove (3) implies (4). Assume
that for every $g\in G\setminus N,$ $\chi(g)=0$ for all $\chi \in
\rm{Irr}(G\mid M).$ Hence, all the generators of $V(G\mid M)$ are
contained in $N.$ Thus, $V(G\mid M)\le N$ as desired. Now, we show
that (4) implies (1). Assume that $ V(G|M)\leq N$ and let $x\in
G\setminus N$ and $ y\in M$. Hence, $yx\not\in N$. Thus $yx\not\in
V(G|M)$. So, for any $\chi_{i}\in \rm{Irr}(G|M)$, $ \chi_{i}(x)=
\chi_{i}(yx)= 0$. Recall that $ \rm{Irr}(G)\setminus
\rm{Irr}(G|M)=\rm{ Irr}(G/M).$ Write $\rm{Irr} (G/M)
=\{\overline{\phi}_{1},\cdots, \overline{\phi}_{r}\}.$ For each $
\overline{\phi}_{i}$ there exists $\phi_{i}\in \rm{Irr}(G)\setminus
\rm{Irr}(G\mid M)$ such that $\phi_{i}(x) = \overline{\phi}_{i}(Mx)
= \phi_{i}(yx)$. Hence, $x$ and $yx$ have the same character values
for all irreducible characters of $G$. Since the irreducible
characters form a basis for the class functions, all class functions
have the same value on $x$ and $xy$. This implies that $x$ and $xy$
are in the same class. Hence, $x$ is conjugate to all of $xM.$ We
conclude that $(G,N,M) $ is a Camina triple. Thus, (4) implies (1).

To finish the proof of the theorem, it is enough to show that (1) is
equivalent to (5). First assume that $(G,N,M)$ is a Camina triple;
that is, if $g\in G\setminus N,$ then $g$ is conjugate to all of
$gM.$ Hence, if $z\in M,$ then there exists $y\in G$ such that
$y^{-1}gy=gz.$ It follows that $ g^{-1}y^{-1}gy=z.$ Conversely,
suppose that for all $g\in G\setminus N$ and $z\in M$ there exists
$y\in G$ such that $[g,y]= z.$ Fix $g\in G\setminus N$ and $z\in M.$
We need to show that $g$ is conjugate to $gz.$ we know there exists
$y$ such that $ g^{-1}y^{-1}gy=z.$ This implies that  $y^{-1}gy=gz.$
Hence, $g$ is conjugate to every element in $gM,$ and $(G,N,M)$ is a
Camina triple as required.
\end{proof}

The following lemma describes the relationship between two Camina
triples in the same group.
\begin{lemma}
If $(G,N_{1},M) $ and $(G,N_{2},M) $ are Camina triples, then
$(G,N_{1}\cap N_{2},M)$ is a Camina triple.
\end{lemma}
\begin{proof}
Notice that $1 < M \leq N_{1} \cap N_{2}$. If $g \in G \setminus
N_{1} \cap N_{2}$, then either $g \in G \setminus N_{1}$ or $g \in G
\setminus N_{2}$. In either case, $g$ is conjugate to all of $gM$.
Hence, $(G, N_{1} \cap N_{2}, M)$ is a Camina triple as desired.
\end{proof}
We now show that Camina pairs are special cases of Camina triples.
\begin{lemma} \label{sixthree}
The triple $ (G, N, N)$ is a Camina triple if and only if $(G,N)$ is
a Camina pair.
\end{lemma}
\begin{proof}
Observe that $(G,N)$ is a Camina pair if and only if for every $g
\in G\setminus N $, we have $ cl(g) = gN $. This occurs if and only
if $ (G, N, N)$ is a Camina triple.
\end{proof}

We now prove a fact about the center of a group $G$ in the case when
$(G,N,M) $ is a Camina triple. Note that it is not difficult to see
that the intersection of $\bf{Z}(G)$ and the set of elements in
$G\setminus N$ has to be the empty set.
\begin{lemma} \label{sixseven}
If $(G,N,M)$ is a Camina triple, then the following are true:
\begin{enumerate}
\item $\bf{Z}(G)\leq N$ and
\item if $K \lhd G$ and $ K< M$, then $(G/K , N/K, M/K)$ is a Camina triple.
\end{enumerate}
\end{lemma}
\begin{proof}
If $ g \in \bf{Z}(G)$, then $g$ is only conjugate to itself. Hence
$g $ is not conjugate to all of $gM$, and so $g \in N$. Therefore
$Z(G)\leq N$. Now, let $K \lhd G$, with $ K< M.$ Hence,  $1 < M/K
\leq N/K < G/K$. Since every $\chi \in \rm{Irr}(G|M)$  vanishes on
$G\setminus N,$ every $\chi \in \rm{Irr}(G/K|M/K)$ vanishes on $G/K
\setminus N/K$. It follows that $(G/K , N/K, M/K)$ is a Camina
triple as desired.
\end{proof}

Next, consider the terms of the upper central series of $G$ when
$(G,N,M)$ is a Camina triple. Let $ Z_{1} =\bf{Z}(G) $ and
$Z_{i}/Z_{i-1}= \textbf{Z}(G/Z_{i-1})$ for $i>1$.
\begin{lemma} \label{sixeight}
If  $(G,N,M)$ is a Camina triple, and $Z_{m} < M$, then $ Z_{m+1}
\leq N$.
\end{lemma}
\begin{proof}
By Lemma \ref{sixseven} part 2, $(G/Z_{m}, N/Z_{m}, M/Z_{m})$ is a
Camina triple. So applying Lemma \ref{sixseven} part 1 to $G/Z_{m}$,
we get $\textbf{Z}(G/Z_{m})\leq N/Z_{m}$. Hence $Z_{m+1} \leq N$ as
desired.
\end{proof}
Now, we need to state this very useful theorem, which is Theorem D
in \cite{Isaacs3}, due to Berkovich.
\begin{theorem}
Let $N$ be a normal subgroup of $G$ and suppose that every member of
cd$(G\mid N')$ is divisible by some fixed prime $p.$ Then $N$ is
solvable and has a normal $p$-complement.
\end{theorem}

We need the next lemma to prove the remaining parts of our first
theorem.
\begin{lemma}\label{sixthirteen}
If $(G, N, M)$ is a Camina triple, then $M$ is solvable and  has a
normal $p$-complement for every prime $p$ that divides $\mid
G:N\mid$.
\end{lemma}
\begin{proof}
Let $\chi\in \rm{Irr}(G\mid M).$ We know by Lemma \ref{sixtwo} that
$\chi(g)=0$ for all $g\in G\setminus N.$ By the discussion in
\cite{Isaacs1} page 200. we deduce that for every prime $p$ divisor
of $| G:N|$, $p$ divides $ \chi(1)$ for all $\chi \in \rm{Irr}(G\mid
M')$. So by Berkovich's theorem, $M$ is solvable and $M$ has a
normal $p$-complement.
\end{proof}
Now, we show that if $(G,N,M)$ is a Camina triple, then $M$ has a
normal $\pi$-complement, where $\pi$ is the set of primes that
divide $|G:N|.$ This proves the remaining parts of Theorem 1.
\begin{lemma}\label{sixfourteen}
If $(G, N, M)$ is a Camina triple, and $ \pi = \{ p$ prime $\mid p$
divides $ | G:N | \}$, then $M$ has a normal $\pi$-complement $Q$
such that $ M/Q$ is nilpotent.
\end{lemma}
\begin{proof}
Since $(G, N, M)$ is a Camina triple, by Lemma \ref{sixthirteen}, we
know that $ M$ has a normal $p$-complement for every $p \in \pi$.
Now, let $ Q$ be the intersection of these normal $p$-complements.
Hence, $Q$ is a normal $\pi$-complement of $M$. Now, to prove that $
M/Q$ is nilpotent, it will be enough to show that any finite group
having a normal $p$-complement for every prime $p$ is nilpotent. Let
$G$ be a finite group that has a normal $p$-complement for every
prime $p$. We work by induction on $|G|$. If $|G| = 1,$ then the
result is trivial, so we may assume $G> 1.$ Let $p$ be a prime, and
we show that G has a normal Sylow $p$-subgroup.  If $G$ is a
$p$-group, then this is trivial. Thus, we may assume that $ |G|$ is
divisible by some prime $q$ which is not $p$. By hypothesis, G has a
normal $q$-complement N. Observe that $N < G,$ and so the induction
hypothesis implies that $N$ is nilpotent. Hence, $N$ has a normal
Sylow $p$-subgroup $P$ and thus, $P$ is characteristic in $N$. But
$G/N$ is a $q$-group, so $P$ is a Sylow $p$-subgroup of $G$.  It
follows that $G$ has a normal Sylow $p$-subgroup as required.
\end{proof}

The following lemma, is very useful.

\begin{lemma} \label{nine}
Let $(G,N,M)$ be a Camina triple. If $x\in G\setminus N$ and $
o(x)=m$, and $y \in C_{M}(x)$ then the order of $y $ divides $m$.
\end{lemma}

\begin{proof}
Since $xy \in xM$, we know that $xy$ is a conjugate to $x$. Thus
$xy$ has order $m$. Hence $x^{m}y^{m} =1$, and so, $y^{m}=1$ as
desired.
\end{proof}

We show in the next lemma that if $G/N$ is not a $p$-group for any
prime $p,$ then $M\cap Z(G)=\{1\}.$

\begin{lemma}\label{nineee}
Let $(G, N, M)$ be a Camina Triple, and $G/N$ is not a $p$-group for
any prime $p$, then $ M\cap Z(G) =\{1\}$.
\end{lemma}

\begin{proof}
If $G/N$ is not a $p$-group, then we can find $x\in G\setminus N$
such that $o(Nx)=p^{a}$, and $y\in G\setminus N$ such that
$o(Ny)=q^{b}$, where $p, q$ are two distinct primes. Let $n$ be the
$p'$-part of the order of $x,$ and $m$ be the $q'$-part of the order
of $y.$ Notice that $x^{n} \not\in N$ and $y^{m} \not\in N.$ Also,
the order of $x^{n}$ is $p^{\alpha}$ and the order of $y^{m}$ is
$q^{\beta}.$ Hence, $C_{M}(x^{n})$ is a $p$-group and $C_{M}(y^{m})$
is a $q$- group. We know that $ M \cap Z(G) \subseteq
C_{M}(x^{n})\cap C_{M}(y^{m})=\{1\}$ as desired.
\end{proof}

In the next result, we prove that if $G$ is nilpotent, then $G/N$
and $M$ are $p$-groups for the same prime $p.$

\begin{lemma}\label{ninee}
Let $(G, N, M)$ be a Camina Triple, if $G$ is nilpotent then $M$ and
$G/N$ are $p$-groups for some prime $p$.
\end{lemma}

\begin{proof}
Since $G$ is nilpotent, then $Z(G) $  cannot intersect with $M$
trivially. Hence, by Lemma \ref{nineee} $G/N$ is $p$-group for some
prime $p$. Now let $x \in G\setminus N$ where $o(Nx)=p^{a},$ Let $n$
be the $p'$-part of the order of $x.$ Notice that $x^{n} \not\in N$
and the order of $x^{n}$ is $p^{\alpha}.$ Hence, $C_{M}(x^{n})$ is a
$p$-group. Thus, $ M \cap Z(G)$ is a $p$-group. Now, suppose that
there exists a prime $q \neq p$ such that $q$ divides $|M|.$ Hence,
there exists $y \in M$ where the order of $y$ is $q^{m}.$ Since $G$
is nilpotent and $(o(x^{n}),o(y))=1.$ Then $y\in C_{G}(x^{n}).$ But,
by Lemma \ref{nine}, we know that $o(y)$ divides the order of
$x^{n}.$ Which leads to a contradiction, and $M$ is a $p$-group.
\end{proof}

We are now ready to prove theorem 2.

\begin{proof}[proof of Theorem 2]
If $G$ is nilpotent, then by Lemma \ref{ninee} we have $(1)$ and
$(2)$ hold. So, we may assume that $G$ is not nilpotent. If $G/N$ is
a $p$-group, then $(1)$ holds. Assume that $G/N$ is not a $p$-group,
and let $\pi=\{p:$ prime such that $p $ divides $ |G/N|\},$ by Lemma
\ref{sixfourteen}, $M$ has a normal $\pi$-complement $Q$ such that $
M/Q$ is nilpotent. If $Q\neq \{1\}$ and proper in $M$, then (2)
holds. Also, if $Q=\{1\},$ then $M$ is nilpotent and $(2)$ holds.
Now, if $M=Q,$ then $(|M|, |G/N|)=1.$ In this case, we know that
maybe $M$ does not have a nilpotent quotient. If $M$ is abelian,
then (4) holds. So, we may assume that $M$ is not abelian. By
Theorem 1, we know that if $(G,N,M)$ is a Camina triple, then $M$ is
solvable. Hence, if $M$ is not abelian, then we can consider $K$ to
a maximal normal subgroup of $M$ such that $M/K$ is not abelian.
Note that, $(M/K)'$ is the unique minimal normal subgroup of $M/K.$
Therefore, the group $M/K$ satisfies the hypothesis of Theorem 12.3
of \cite{Isaacs1}. So, either $M/K$ is a $p$-group and hence $M/K$
is nilpotent and (3) holds, or $M/K$ is a Frobenius group with an
abelian Frobenius complement, and $(M/K)'$ is the Frobenius kernel
and is an elementary abelian $p$-group and (4) holds as desired.

\end{proof}

\section{Camina pairs}

We are now prove some results about Camina pairs using Camina
triples results, and given the fact that they are special cases of
Camina triples. In \cite{camina1}, Camina defined a different
hypothesis that is equivalent to Camina pairs. Let $G$ be a finite
group with a proper normal subgroup $N \neq 1$ and a set of
irreducible non-trivial characters of $G$, $A=\{
\chi_{1}\cdots\chi_{n}\},$ where $n$ is a natural number, such that
\begin{enumerate}
\item $\chi_{i}$ vanishes on $ G\setminus N$ and
\item there exist natural numbers $\alpha_{1}\cdots\alpha_{n} > 0$
such that $\sum_{i=1}^{n}\alpha_{i}\chi_{i}$ is constant on
$N\setminus \{1\}.$
\end{enumerate}
We are able to identify the characters in Camina hypothesis in
\cite{camina1}. First, let $N$ be a normal subgroup of $G$ and
$\theta \in \rm{Irr}(N).$ The inertia group of $\theta$ in $G$
denoted by $T$ and defined by $\{g\in G\mid \theta^{g}=\theta\}.$
\begin{theorem} \label{sixseventeen} Let
$(G,N) $ be a Camina Pair then, $A = \rm{Irr}(G|N)$.
\end{theorem}
\begin{proof}
First, we show that $ A\subseteq \rm{Irr}(G|N).$ To see this,
suppose $\chi_{j} \in A \setminus  \rm{Irr}(G | N).$ This implies
that $\chi_{j} \in \rm{Irr}(G/N)$. On the other hand, since
$\chi_{j} \in A$, we have that $\chi_{j} (x) = 0$ for all $x \in G
\setminus N$.  This implies $\chi_{j} (xN) = 0$ for all $xN \in G/N
\setminus  \{ N \}$, and hence, $\chi_{j}$ is a multiple of the
regular character of $G/N$. Since $N < G$, we know that the regular
character of $G/N$ is not irreducible, and so, we have a
contradiction since it is not possible for an irreducible character
to be a multiple of a reducible character. 
Thus no such $\chi_{j}$ exist in $A$. Therefore $
A\subseteq \rm{Irr}(G|N)$. 
On the other hand, for every $ 1_{N} \neq \theta \in\rm{Irr}(N) $
and by Theorem 6.11 in \cite{Isaacs1}, there exist $ \chi_{i} \in A
$ such that $ \chi_{i} \in\rm{Irr}(G|\theta)$. Notice that
$\theta^{G}(g)=0$ if $g \not\in N$, and if $ g \in N$, then  $
\theta^{G}(g)= \frac{1}{|N|}\sum_{x\in G} \theta^{x}(g)$, hence
$\theta^{G} (g) = \frac{1}{|N|} |T|(\theta_1 (g) + \dots + \theta_n
(g))= \theta^{G} (g) = |T:N| (\theta_1 (g) + \dots + \theta_n (g))$
where $\theta_{i},  i=1\dots n,$ are the distinct conjugates of
$\theta$ in $G$. Note that $ \chi_{i}(g) =0 $ if $ g\not\in N$, and
if $ g \in N$, $ \chi_{i}(g) = a (\theta_{1}(g)+\dots+
\theta_{n}(g))$ where, $a$ is a non negative integer. Hence
$\chi_{i}= c \theta^{G}$. Thus, $ \chi_{i}$ is the unique
irreducible constituent of $\theta^{G}$. Thus $|\rm{Irr}(G|\theta)|
=1$, and $ \rm{Irr}(G|\theta) \subseteq A$ and since $
\rm{Irr}(G|N)=\cup_{1\neq \theta \in \rm{Irr}(N)}
\rm{Irr}(G\mid_{\theta})$. Then $ |A|= |\rm{Irr}(G|N)|$. And since $
A\leq\rm{Irr}(G|N)$ then, $A = \rm{Irr}(G|N)$ as desired.
\end{proof}
Our last result in this chapter, states some new conditions for a
pair $(G,N)$ to be a Camina pair.
\begin{theorem} \label{sixeighteen}
Let $G$ be a finite group and $ N\vartriangleleft G$, then the
following are equivalent:
\begin{enumerate}
\item $(G,N)$ is a \emph{CP}.
\item $V(G|N)$ = $N$.
\item There is no $x$ in $N$ such that $\chi (x) =0 $ for all $
\chi$'s in $ \rm{Irr}(G|N)$, and if $ x \in G\setminus N$, then
$\chi(x) =0 $ for all $ \chi$'s in $ \rm{Irr}(G|N)$
\end{enumerate}

\end{theorem}

\begin{proof}
Notice that 3) implies 2)is trivial. To prove 2) implies 1), assume
that $ V(G|N) =N $, by Theorem \ref{sixtwo}, $( G,N,N)$ is a Camina
triple. Thus, by Lemma \ref{sixthree}, $(G,N)$ is a Camina pair. To
prove 1) is equivalent to 3), by Lemma \ref{sixthree} and Theorem
\ref{sixtwo}, $(G,N)$ is a Camina pair if and only if $V(G\mid
N)\le N.$
\end{proof}

\end{document}